\documentclass[12pt,a4]{amsart}
\usepackage{psfrag}
\usepackage{graphicx}
 \usepackage{righttag}
\usepackage{amsmath}
\usepackage{enumerate}
\usepackage{amssymb}
\usepackage{amsbsy}
\usepackage{amsfonts}
\theoremstyle{plain}
\newtheorem{thm}{Theorem}[section]
\newtheorem{prop}[thm]{Proposition}
\newtheorem{lem}[thm]{Lemma}
\newtheorem{cor}[thm]{Corollary}

\theoremstyle{definition}

\newtheorem{rmk}[thm]{Remark}
\numberwithin{equation}{section}
\newcommand{\sm}{\left(\begin{smallmatrix}}
\newcommand{\esm}{\end{smallmatrix}\right)}

\newfont{\FieldFont}{msbm10 scaled\magstep1}

\renewcommand{\theequation}{\arabic{section}.\arabic{equation}}

\begin{document}
\title[Functional equations for double series]
{Functional equations for double series of Euler type with coefficients}

\author{YoungJu Choie }

 \address{Department of Mathematics and PMI\\
 Pohang University of Science and Technology\\
 Pohang, 790--784, Korea}
 \email{yjc@postech.ac.kr}

 \author{Kohji Matsumoto}

 \address{Graduate School of Mathematics  \\
 Nagoya University \\
 Chikusa-ku, Nagoya 464-8602, Japan}
  \email{kohjimat@math.nagoya-u.ac.jp}

 \thanks{This work
 was partially supported by NRF 2012047640, NRF 2011-0008928 and  NRF 2008-0061325
}

\begin{abstract}
We prove two types of functional equations for double series of Euler type
with complex coefficients.   The first one is a generalization of
the functional equation for the Euler double zeta-function, proved in a former
work of the second-named author.
The second one is more specific, which is proved when the coefficients are Fourier
coefficients of cusp forms and the modular relation is essentially used in the
course of the proof. As a consequence of functional equation we are able to determine trivial zero divisors.
\end{abstract}
 \maketitle

\section{\bf{Introduction}}\label{sec-0}

Inspired by two source, namely,
the theory of multiple zeta values on the one hand,
and the theory of modular symbols
and  periods of cusp forms on the other,
Manin in \cite{M-iterate},
\cite{M-S} extended the theory of periods of modular forms
replacing integration along geodesics in
the complex upper half plane
by iterated integrations to set up the foundation of  
the theory of "iterated
noncommutative modular symbols". 
In particular
Manin\cite{M-iterate, M-S} 
considered the following iterated Mellin transform  
$$I_{i\infty}^{0}(\omega_{s_{\ell}},..,\omega_{s_{1}}):
=\int_{i\infty}^0 \omega_{s_{\ell}}( \tau_{\ell})\int_{i\infty}^{\tau_{\ell}}
\omega_{s_{\ell-1}}(\tau_{\ell-1})..\int_{i\infty}^{\tau_{2}}\omega_{s_{1}}(\tau_{1})$$
of a
finite sequence of cusp forms $f_1,...,f_{\ell}$ 
of weight $k_j \in \mathbb{N}$ with respect
to a congruence subgroup $\Gamma$ of $ SL_2(\mathbb{Z})$ and  $\omega_{s_j}(\tau):=f_j(\tau)\tau^{s_j-1}d\tau ,  s_j\in \mathbb{C}, j=1,...,\ell.$
When $\ell=1$ then
$$I_{i\infty}^0(\omega_{s})
=\int_{i\infty}^0 f(\tau)\tau^{s-1} d\tau$$ is
the classical Mellin transform of
a cusp form $f\in S_k(\Gamma)$ satisfying
the following functional equation:

$$ I_{i\infty}^0(\omega_{s}) = 
-\epsilon_f e^{\pi i s} 
N^{\frac{k}{2}-s}I_{i\infty}^{0}(\omega_{ k-s})$$
if $f$ is an eigenform with eigenvalue $\epsilon_f = \pm 1$
with respect to the involution
$\omega_N=\sm 0 & -1 \\ N & 0 \esm$ (see \cite{Miy}).

However it seems that it is not 
anymore true to expect a simple 
functional equation if $\ell \geq 2.$
Manin \cite{M-S} said
"since a neat functional equation can be
written not for these individual integrals but
for their generating series..," so   the
 functional equation of the "total Mellin transform" associated to the finite family 
  $\{f_j|j=1,..,\ell,..\}$ of cusp forms
was derived.\\

Now consider the case when $\ell=2 $:
for $ f_j(\tau)=\sum_{n\geq 1}a_j(n)e^{2\pi i n \tau},j=1,2$, we have

$$I_{i\infty}^{0}(\omega_{s_{2}}, \omega_{s_{1}})=\int_{i\infty}^{0}
f_2(\tau_2){\tau_2}^{s_2-1}
\int_{i\infty}^{\tau_2}f_1(\tau_1) {\tau_1}^{s_1-1}d\tau_1 d\tau_2$$
$$=\int_{i\infty}^{0}
f_2(\tau_2){\tau_2}^{s_2-1}d\tau_2
\int_{i\infty}^{0}f_1(\tau_1+\tau_2) (\tau_1+\tau_2)^{s_1-1}d\tau_1.
$$
If $s_1\geq 2$ is a  positive integer, then
$I_{i\infty}^{0}(\omega_{s_2}, \omega_{s_1})$ is a finite linear combination
of the following multiple Dirichlet series:
$$
 \sum_{n,m\geq 1} \frac{a_1(n)a_2(m)}{(n+m)^{s_2}m^{r}}, 
 \quad 0\leq r\leq s_1,
 $$
where $s_2\in\mathbb{C}$ whose real part is sufficiently large. 

In this paper we   study more general type of  multiple Dirichlet series,
motivated by the above iterated Mellin transform of Manin.
Take  ${\frak A}=\{a(n)\}_{n\geq 1}$ be a sequence of complex numbers, and define
\begin{equation} \label{A-zeta}
L_2(s_1,s_2;{\frak A})=\sum_{m,n\geq 1} \frac{a(n)}{m^{s_1}(m+n)^{s_2}},
\end{equation}
where $s_j=\sigma_j+it_j$ ($j=1,2$) be two complex variables.   The purpose of
the present paper is to prove two types of functional equations for this double
series.
This is more general situation since we take $a_1(n)$ be an arbitrary complex number
and allow $a_2(n)=1$ for any $n\geq 1.$

Before stating our main results, we recall functional equations for classical
zeta-functions, and for the double zeta-function without coefficients.
It is well-known that the Riemann zeta-function $\zeta(s)$
($s=\sigma+it\in\mathbb{C}$) has the beautiful
symmetric functional equation
\begin{align}\label{0-riemann-fe}
\pi^{-s/2}\Gamma\left(\frac{s}{2}\right)\zeta(s)=
\pi^{-(1-s)/2}\Gamma\left(\frac{1-s}{2}\right)\zeta(1-s).
\end{align}
For Hurwitz zeta-functions $\zeta(s,\alpha)=\sum_{n\geq 0}(n+\alpha)^{-s}$
($\alpha>0$),
however, this symmetricity is no longer valid in general.
In fact, the functional equation for $\zeta(s,\alpha)$ is of the form
\begin{align}\label{0-hurwitz-fe}
\zeta(s,\alpha)=\frac{\Gamma(1-s)}{i(2\pi)^{1-s}}\left\{e^{\pi is/2}
\phi(1-s,\alpha)-e^{-\pi is/2}\phi(1-s,-\alpha)\right\},
\end{align}
where
$\phi(s,\alpha)=\sum_{n\geq 1}e^{2\pi in\alpha}n^{-s}$ is the Lerch
zeta-function (Titchmarsh \cite[(2.17.3)]{Tit}).
See also a recent work of Lagarias and Li \cite{LL}.
Those functional equations are very fundamental in the study of $\zeta(s)$ and
$\zeta(s,\alpha)$.

The theory of multiple zeta-functions has been studied very actively in recent
decades.    Searching for some kind of functional equations is a quite natural
problem in order to develop further analytic studies of multiple zeta-functions.
In the simplest double zeta case, the following functional equation is already
known: consider the case ${\frak{A}_1}:=\{a(n)=1\; {\rm for\;all}\; n\}$
in the definition \eqref{A-zeta}, that is
\begin{equation} \label{orig-zeta}
\zeta_2(s_1,s_2)=\sum_{m,n\geq 1} \frac{1}{m^{s_1}(m+n)^{s_2}}.
\end{equation}
This is called the Euler double zeta-function,
and satisfies the following
functional equation:
\begin{align}\label{0-double-fe}
\zeta_2(s_1,s_2)=\frac{\Gamma(1-s_1)\Gamma(s_1+s_2-1)}{\Gamma(s_2)}\zeta(s_1+s_2-1)\\
+\Gamma(1-s_1)\left\{F_+(1-s_2,1-s_1; \frak{A}_1)+F_-(1-s_2,1-s_1; \frak{A}_1)\right\},\notag
\end{align}
where
\begin{equation}\label{0-Fpm}
F_{\pm}(s_1,s_2; \frak{A}_1)=\sum_{k\geq 1}\sigma_{s_1+s_2-1}(k)\Psi(s_2,s_1+s_2;
\pm 2\pi ik),
\end{equation}
with $\sigma_c(k)=\sum_{0<d|k}d^c$ and
\begin{align}\label{hyp-def}
\Psi(a,b;x)=\frac{1}{\Gamma(a)}\int_0^{e^{i\phi}\infty}e^{-xy}y^{a-1}
    (y+1)^{b-a-1}dy
\end{align}
(the confluent hypergeometric function),
where $\Re a>0$, $-\pi<\phi<\pi$, and $|\phi+\arg x|<\pi/2$.

Formula \eqref{0-double-fe} may be regarded as a double analogue of \eqref{0-hurwitz-fe}.
In fact, since the asymptotic expansion
\begin{align}\label{asymp}
\Psi(a,b;x)=\sum_{j=0}^{M-1}\frac{(-1)^j(a-b+1)_j(a)_j}{j!}x^{-a-j}
    +\rho_M(a,b;x)
\end{align}
(where $(a)_j=\Gamma(a+j)/\Gamma(a)$ and $\rho_M(a,b;x)$ is the remainder term;
see \cite[formula 6.13.1(1)]{Er}) is known, taking the first term of
the right-hand side of \eqref{asymp}, we can ``approximate''
$F_{\pm}(s_1,s_2; \frak{A}_1)$ by the Dirichlet series
$\sum_{k\geq 1}\sigma_{s_1+s_2-1}(k)(\pm 2\pi ik)^{-s_2}$.   Therefore $F_{\pm}(s_1,s_2; \frak{A}_1)$
is a kind of ``generalized Dirichlet series''.

Moreover, \eqref{0-double-fe} gives a symmetric form of functional equation,
similar to \eqref{0-riemann-fe}, on some hyperplanes (see Remark \ref{rmk-symm}
at the end of Section \ref{sec-4}).

Formula \eqref{0-double-fe} was essentially
included in \cite{Mat04}, and first explicitly stated in \cite{Mat-Sugaku}
(in a generalized form with certain shifting parameters).
In \cite{KMT-Debrecen} a generalization of formula \eqref{0-double-fe} to the case
where the denominator includes certain complex parameters.   Moreover in
\cite{KMT-IJNT} a formula analogous to \eqref{0-double-fe} was shown for double
$L$-functions whose numerator includes Dirichlet characters.

The purpose of the present paper is to discuss such kind of
functional equations in a more general setting. As a consequence we are able to 
determine trivial zero divisors of double series.
  We will state the
main results in the next section.

\section{\bf{Statement of results}}\label{sec-0.5}
Our first main result in the present paper is a further generalization of formula
\eqref{0-double-fe}.
We assume that ${\frak A}=\{a(n)\}_{n\geq 1}$
be a sequence of complex numbers
satisfying\\

\begin{enumerate}
\item[(i)]  $a(n)\ll n^{(\kappa-1)/2+\varepsilon}$
with a certain constant $\kappa\geq 1$,
where $\varepsilon$ is an arbitrarily small positive number,

\item[(ii)] the Dirichlet series $L(s,{\frak A})=\sum_{n\geq 1}a(n)n^{-s}$
(which is absolutely convergent for $\Re(s) >(\kappa+1)/2$ by (i))
can be continued to the whole complex plane as a meromorphic function which has only finitely many poles.

\end{enumerate}

\medskip

Let $\mathcal{H}$ be the complex upper half plane and let
\begin{equation}\label{form}
f(\tau)=\sum_{n\geq 1}a(n)q^n,
\end{equation}
where $q=e^{2\pi i \tau}, \tau \in \mathcal{H}$.
It is obvious by (i) that $f(\tau)$ is convergent for $\tau\in\mathcal{H}$
and holomorphic in $\tau$.
Moreover, using (i) we find that the right-hand side of \eqref{A-zeta} is
\begin{align*}
&\ll \sum_{m,n\geq 1}n^{(\kappa-1)/2+\varepsilon}
   m^{-\sigma_1}(m+n)^{-\sigma_2}\\
&\leq \sum_{m,n\geq 1}m^{-\sigma_1}
   (m+n)^{(\kappa-1)/2+\varepsilon-\sigma_2},
\end{align*}
so, using \cite[Theorem 3]{Mat-Millennial}, we see that \eqref{A-zeta} is convergent
absolutely in the region
\begin{align}\label{region1}
\sigma_2>\frac{\kappa+1}{2},\quad
\sigma_1+\sigma_2>\frac{\kappa+3}{2}.
\end{align}
Under ussumption (ii), using \cite{MT} we can show that \eqref{A-zeta} has meromorphic
continuation to the whole complex space $\mathbb{C}^2$.

Let
\begin{align}\label{A-def2}
A_c(l)=\sum_{0<n|l}n^c a(n),
\end{align}
and
\begin{align}\label{def-F2}
F_{\pm}(s_1,s_2;{\frak A})=
\sum_{l\geq 1}A_{s_1+s_2-1}(l)\Psi(s_2,s_1+s_2;\pm 2\pi il).
\end{align}
Also put
\begin{align}\label{def-G}
L_1(s_1,s_2;{\frak A})=
\frac{\Gamma(1-s_1)\Gamma(s_1+s_2-1)}{\Gamma(s_2)}L(s_1+s_2-1;{\frak A}).
\end{align}

\begin{thm}\label{thm-1}
{\rm (}The first form of the functional equation{\rm )}
Under the above assumptions {\rm (i)} and {\rm (ii)},
The functions $F_{\pm}(s_1,s_2;{\frak A})$ can be continued meromorphically to the whole
space $\mathbb{C}^2$, and
for any $s_1,s_2\in\mathbb{C}$, except for singularity points, it holds that
\begin{align}\label{thm-formula}
&L_2(s_1,s_2;{\frak A})=L_1(s_1,s_2;{\frak A})\\
&\quad+\Gamma(1-s_1)\left\{F_+(1-s_2,1-s_1;{\frak A})+F_-(1-s_2,1-s_1;{\frak A})
\right\}.
\notag
\end{align}
\end{thm}

\begin{rmk}
We can determine the location of singular locus of $L_2(s_1,s_2;{\frak A})$ from the
right-hand side of \eqref{thm-formula}.   In fact, the explicit form of
$L_1(s_1,s_2;{\frak A})$ is given by \eqref{def-G}, while the explicit
information on the singular locus of $F_{\pm}$ can be obtained from
\eqref{asymp5} and \eqref{rho-2} in Section \ref{sec-4}.
\end{rmk}

The proof of Theorem \ref{thm-1}, which will be described in Sections \ref{sec-1}
to \ref{sec-4}, is analogous to that in \cite{Mat04}, the basic idea of
which goes back to Motohashi \cite{Mot} and Katsurada and Matsumoto \cite{KM}.

Since the assumptions (i) and (ii) for ${\frak A}$ is very general, we may discuss
various specific examples.  For instance, by replacing  $\frak{A}$ by  $\frak{A}_1$
in (\ref{thm-1})
we recover the functional equation (\ref{0-double-fe}) of
double zeta function.
Further let us consider the very special situation that
$$
\frak{A}_0(n)=\{a(n)=1  \mbox {\, for only one fixed $n$, and $a(n)=0$ for all other
$n$}\}.
$$
In this case (\ref{thm-1}) is reduced to
\begin{align}\label{specialcase}
L_2(s_1,s_2;{\frak A}_0(n))=\sum_{m\geq 1}\frac{1}{m^{s_1}
    (m+n)^{s_2}},
\end{align}
a single series in two variables.   This is a special case of the series
$$
\xi(s_1,s_2;(\alpha,\beta)):=\sum_{m\geq 0}\frac{1}
    {(m+\alpha)^{s_1}(m+\beta)^{s_2}}\qquad(\beta\geq\alpha>0),
$$
which was used in \cite{Mat03}.   From the above theorem we immediately obtain
the following ``two-variables analogue'' of \eqref{0-hurwitz-fe}.

\begin{cor}\label{cor}
For any $s_1,s_2\in\mathbb{C}$, except for singularity points, it holds that
\begin{align}\label{cor-formula}
L_2(s_1,s_2;{\frak A}_0(n))&
=\frac{\Gamma(1-s_1)\Gamma(s_1+s_2-1)}{\Gamma(s_2)} \cdot \frac{1}{ n^{s_1+s_2-1}}\\
&+\Gamma(1-s_1)\{F_+(1-s_2,1-s_1;\frak{A}_0(n))+
F_-(1-s_2,1-s_1;\frak{A}_0(n))\},\notag
\end{align}
where
$$
F_{\pm}(s_1,s_2;\frak{A}_0(n))=n^{s_1+s_2-1}\sum_{k\geq 1}\Psi(s_2,s_1+s_2;\pm 2\pi ikn).
$$
\end{cor}

\begin{rmk}
It is to be noted that
$$L_2(s_1,s_2; \frak{A})= \sum_{n\geq 1}a(n)L_2(s_1,s_2; \frak{A}_0(n)).$$
Therefore, multiplying the both sides of \eqref{cor-formula} by $a(n)$ and adding
with respect to $n$, we obtain \eqref{thm-formula}.
From this observation we may say that Corollary \ref{cor} is a ``refinement'' or "decomposition" of
Theorem \ref{thm-1}.
\end{rmk}

Another important example, closely related with our original motivation on periods,
is the case that ${\frak A}$ is the set of Fourier
coefficients of a certain cusp form.   Now assume that \eqref{form} is
a holomorphic cusp form
of weight $\kappa$ with respect to the Hecke congruence subgroup $\Gamma_0(N)$.
In this case the assumptios (i) and (ii) are surely satisfied; (i) is Deligne's
estimate and $\kappa$ is the weight.
In this case we write $L_2(s_1,s_2;{\frak A})$,
$L_1(s_1,s_2;{\frak A})$, $L(s,{\frak A})$ and $F_{\pm}(s_1,s_2;{\frak A})$ by
$L_2(s_1,s_2;f)$,
$L_1(s_1,s_2;f)$, $L(s,f)$ and $F_{\pm}(s_1,s_2;f)$, respectively.
Then \eqref{thm-formula} can be written as
\begin{align}\label{thm-formula-cusp}
&L_2(s_1,s_2;f)=L_1(s_1,s_2;f)\\
&\quad+\Gamma(1-s_1)\left\{F_+(1-s_2,1-s_1;f)+F_-(1-s_2,1-s_1;f)\right\}.
\notag
\end{align}
Since this formula is proved under the above very general setting, no property of
cusp form is used in the proof.
When $f$ is a cusp form, it is natural to expect some different type of results, for
which the modularity is essestially used.   Our second main result gives such a
functional equation.   Let
\begin{align}\label{f-tilde}
\widetilde{f}(\tau)=\left(f\left|_{\kappa}\; \omega_N\right.\right)(\tau)
=\left(\sqrt{N}\tau\right)^{-\kappa}f\left(-\frac{1}{N\tau}\right).
\end{align}
This $\widetilde{f}$ is again a cusp form of weight $\kappa$ with
respect to $\Gamma_0(N)$, and especially $\widetilde{f}=f$ when $N=1$.
We write the Fourier expansion of $\widetilde{f}$ at $\infty$ as
$\widetilde{f}(\tau)=\sum_{n\geq 1}\widetilde{a}(n)q^n$.
Define
\begin{align}\label{H-2-def}
H_{2,N}^{\pm}(s_1,s_2;\widetilde{f})&=\sum_{m,n\geq 1}m^{-s_1-s_2}\widetilde{a}(n)
   \Psi(s_1+s_2,s_2;\pm 2\pi in/Nm).
\end{align}

\begin{thm}\label{thm-2}
(The second form of the functional equation)
When $f(\tau)$
is a cusp form of weight $\kappa$ with respect to $\Gamma_0(N)$,
the functions $H_{2,N}^{\pm}(s_1,s_2;\widetilde{f})$ can be continued
meromorphically to the whole space $\mathbb{C}^2$, and we have
\begin{align}\label{thm2-formula}
&L_2(s_1,s_2;f)=L_1(s_1,s_2;f)\\
&\;+\frac{(2\pi)^{s_1+s_2-1}\Gamma(1-s_1)}{\Gamma(s_2)}
    \Gamma(\kappa-s_1-s_2+1)\notag\\
&\qquad\times N^{-\kappa/2}\left\{e^{\pi i(1-s_1-s_2)/2}
    H_{2,N}^+(-s_1,\kappa-s_2+1;\widetilde{f})\right.\notag\\
&\;\qquad\left. +e^{\pi i(s_1+s_2-1)/2}
    H_{2,N}^-(-s_1,\kappa-s_2+1;\widetilde{f})\right\}.\notag
\end{align}
\end{thm}

The proof of this theorem will be given in Sections \ref{sec-5} and \ref{sec-6}.

Recall that an important application of the classical functional
equation \eqref{0-riemann-fe} is that from which we can find the
"trivial zeros" at negative even integer points of $\zeta(s)$.
The above Theorem \ref{thm-2} has the same type of application.
In fact, as we will see in Section \ref{sec-7}, we can show the 
following

\begin{cor}\label{cor-2}
For any non-negative integer $l$, the hyperplane $\Re s_2=-l$ is a
zero-divisor of $L_2(s_1,s_2;f)$.
\end{cor}

Note that this corollary cannot be deduced from Theorem \ref{thm-1}.
These zero-divisors may be regarded as "trivial zeros" of
$L_2(s_1,s_2;f)$.

\begin{rmk}\label{rem1.5}
It is to be noted that, when $f(\tau)$ is a cusp form, $L_1(s_1,s_2;f)$ also
satisfies a functional equation.   Let
$$
L_1^*(s_1,s_2;f)=(2\pi)^{-s_1-s_2}\Gamma(s_1)\Gamma(s_2)L_1(s_1,s_2;f).
$$
Then from the functional equation for $L(s,f)$ (see, e.g.,
\cite[Theorem 4.3.6]{Miy}) we can deduce
$$
L_1^*(s_1,\kappa-2s_1-s_2+2;\widetilde{f})=(-1)^{\kappa/2}N^{s_1+s_2-\kappa/2-1}
L_1^*(s_1,s_2;f).
$$
\end{rmk}

\begin{rmk}\label{rem2}
A very different type of functional equation for certain
iterated integrals related with certain multiple Hecke $L$-series has been
proved by \cite{CI}.
\end{rmk}
\bigskip

{\bf{ Acknowledgement }} \, The authors express their sincere gratitude to Professor Takashi Taniguchi for his
valuable comments.

\section{\bf{An integral expression}}\label{sec-1}

Now we assume (i) and (ii) in the introduction and start the proof of
Theorem \ref{thm-1}.
In this section we prove the following integral
expression of $L_2(s_1,s_2;{\frak A})$.

\begin{prop}\label{prop-2}
In the region
\begin{align}\label{region2}
\sigma_1>0,\quad \sigma_2>\frac{\kappa+1}{2},\quad
\sigma_1+\sigma_2>\frac{\kappa+3}{2},
\end{align}
the double integral
\begin{align}\label{Lambda}
\Lambda(s_1,s_2;{\frak A})=\int_{0}^{\infty} f(iy)
\int_{0}^{\infty}\frac{1}{e^{2\pi(x+y)}-1}x^{s_1-1} y^{s_2-1} dxdy
\end{align}
converges, and we have
\begin{align}\label{f-i-e-sp}
L_2(s_1,s_2;{\frak A})
 =\frac{(2\pi)^{s_1+s_2}}{\Gamma(s_1)\Gamma(s_2)}\Lambda(s_1,s_2;{\frak A}).
\end{align}
\end{prop}

\begin{proof}
Put
\begin{align}\label{g-trivial}
g_0(\tau)=\sum_{m\geq 1}q^m=\sum_{m\geq 1}e^{2\pi i\tau m}
    =\frac{e^{2\pi i\tau}}{1-e^{2\pi i\tau}}=\frac{1}{e^{-2\pi i\tau}-1}.
\end{align}
Let $\delta>0$, and at first assume $y\geq\delta$.
Then
\begin{align}\label{f-i-e1}
\int_0^{\infty}g_0(i(x+y))x^{s_1-1}dx
&=\int_0^{\infty}\sum_{m\geq 1}e^{2\pi i\cdot i(x+y)m}x^{s_1-1}dx\\
&=\sum_{m\geq 1}e^{-2\pi ym}\int_0^{\infty}e^{-2\pi xm}x^{s_1-1}dx\notag\\
&=(2\pi)^{-s_1}\Gamma(s_1)\sum_{m\geq 1}\frac{e^{-2\pi ym}}{m^{s_1}}\notag
\end{align}
if the gamma integral converges, that is,
if $\sigma_1>0$ holds.   The change of integration and summation in the above can be
justified by absolute convergence (because $y>0\;$).   Therefore
\begin{align}\label{f-i-e2}
&\int_{\delta}^{\infty}f(iy)\int_0^{\infty}g_0(i(x+y))x^{s_1-1}dx\;y^{s_2-1}dy\\
&\;=(2\pi)^{-s_1}\Gamma(s_1)\int_{\delta}^{\infty}\sum_{n\geq 1}a(n)e^{-2\pi yn}
    \sum_{m\geq 1}\frac{1}{m^{s_1}}e^{-2\pi ym}y^{s_2-1}dy\notag\\
&\;=(2\pi)^{-s_1}\Gamma(s_1)\sum_{m,n\geq 1}\frac{a(n)}{m^{s_1}}
    \int_{\delta}^{\infty}e^{-2\pi y(m+n)}y^{s_2-1}dy,\notag
\end{align}
if we can again change the integration and summation.
The series on the right-hand side is
\begin{align*}
&\leq \sum_{m,n\geq 1}\frac{|a(n)|}{m^{\sigma_1}}
    \int_{\delta}^{\infty}e^{-2\pi y(m+n)}y^{\sigma_2-1}dy\\
&\leq \sum_{m,n\geq 1}\frac{|a(n)|}{m^{\sigma_1}}
    \int_{0}^{\infty}e^{-2\pi y(m+n)}y^{\sigma_2-1}dy\\
&=(2\pi)^{-\sigma_2}\Gamma(\sigma_2)\sum_{m,n\geq 1}\frac{|a(n)|}
    {m^{\sigma_1}(m+n)^{\sigma_2}}
\end{align*}
if $\sigma_2>0$ holds.   The resulting infinite series is convergent if
\eqref{region1} holds.
Therefore, if \eqref{region1} (which includes the conditon $\sigma_2>0$) holds, then
the change of integration and summation
in the course of \eqref{f-i-e2} is justified, and moreover, on the right-hand
side of \eqref{f-i-e2}, we can take the limit $\delta\to 0$ termwisely.
Then the right-hand side tends to
\begin{align}\label{f-i-e3}
(2\pi)^{-s_1-s_2}\Gamma(s_1)\Gamma(s_2)\sum_{m,n\geq 1}\frac{a(n)}
    {m^{s_1}(m+n)^{s_2}},
\end{align}
while the left-hand side of \eqref{f-i-e2} tends to $\Lambda(s_1,s_2;{\frak A})$.
This completes the proof.
\end{proof}

\section{\bf{Separating a single series factor}}\label{sec-2}

The integrand of the inner integral of the right-hand side of \eqref{f-i-e-sp}
is singular at $x+y=0$.   The next step is to ``separate'' the contribution of
this singularity.   Let
\begin{align}\label{h-def}
h(z)=\frac{1}{e^{2\pi z}-1}-\frac{1}{2\pi z}.
\end{align}
Using this function, we rewrite \eqref{f-i-e-sp} as follows:
\begin{align}\label{decompose}
\Lambda(s_1,s_2;{\frak A})&=\int_{0}^{\infty} f(iy) \int_{0}^{\infty}h(x+y)x^{s_1-1}
     y^{s_2-1}dxdy\\
     &\quad+\int_{0}^{\infty} f(iy) \int_{0}^{\infty}\frac{x^{s_1-1}y^{s_2-1}}
     {2\pi(x+y)}dxdy\notag\\
     &=I_1+I_2,\notag
\end{align}
say.    To verify this decomposition, we have to check the absolute convergence
of $I_1$ and $I_2$.    Consider $I_2$ under the condition
\begin{align}\label{region3}
0<\sigma_1<1.
\end{align}
It is known that
\begin{align}\label{beta}
\int_0^{\infty}\frac{x^{s_1-1}}{x+y}dx=y^{s_1-1}\Gamma(s_1)\Gamma(1-s_1)
\end{align}
holds for $0<\sigma_1<1$ and $y>0$.   Therefore, under \eqref{region3}, we have
\begin{align}\label{I-2-1}
I_2&=\lim_{\delta\to 0}\int_{\delta}^{\infty}f(iy)y^{s_2-1}\frac{y^{s_1-1}}{2\pi}
\Gamma(s_1)\Gamma(1-s_1)
            dy\\
  &=\frac{1}{2\pi}\Gamma(s_1)\Gamma(1-s_1)\int_0^{\infty}f(iy)y^{s_1+s_2-2}dy,\notag
\end{align}
if the last integral is convergent.   But the last integral is
\begin{align}\label{I-2-2}
&=\int_0^{\infty}\sum_{n\geq 1}a(n)e^{-2\pi yn}y^{s_1+s_2-2}dy\\
&=\sum_{n\geq 1}a(n)\int_0^{\infty}e^{-2\pi yn}y^{s_1+s_2-2}dy\notag\\
&=(2\pi)^{-s_1-s_2+1}\Gamma(s_1+s_2-1)\sum_{n\geq 1}\frac{a(n)}{n^{s_1+s_2-1}}\notag
\end{align}
if $\sigma_1+\sigma_2>1$,
and the last sum is absolutely convergent if
$\sigma_1+\sigma_2>(\kappa+3)/2$ and is equal to
$L(s_1+s_2-1,{\frak A})$.   This verifies
the change of
integration and summation in the course of \eqref{I-2-2}, and the convergence
of the last integral of \eqref{I-2-1}.   Therefore we obtain
\begin{align}\label{I-2-3}
I_2=(2\pi)^{-s_1-s_2}\Gamma(s_1)\Gamma(1-s_1)\Gamma(s_1+s_2-1)L(s_1+s_2-1,{\frak A})
\end{align}
in the region
\begin{align}\label{region4}
0<\sigma_1<1,\quad \sigma_1+\sigma_2>\frac{\kappa+3}{2}.
\end{align}

As for $I_1$, we first note that $h(z)$ is holomorphic at $z=0$, so it is
$O(1)$ when $|z|$ is small.   If the real part of $z$ is large, then clearly
$h(z)=O(|z|^{-1})$.   Therefore
\begin{align}\label{I-1-1}
I_1&\ll\int_0^{\infty}|f(iy)|\left\{\int_0^1 x^{\sigma_1-1}dx +
   \int_1^{\infty}\frac{x^{\sigma_1-1}}{x+y}dx\right\}y^{\sigma_2-1}dy,
\end{align}
and the quantity in the curly bracket is $O(1)$, uniformly in $y$, if
\eqref{region3} holds.
Under the condition (i) at the beginning of Section \ref{sec-1}, it is known that
\begin{align}\label{I-1-2}
f(iy)\ll\left\{
   \begin{array}{lll}
     y^{-(\kappa+1)/2-\varepsilon} & {\rm as} &  y\to 0,\\
     e^{-2\pi y}          & {\rm as} &  y\to\infty
   \end{array}\right.
\end{align}
(see \cite[Lemma 4.3.3]{Miy}).   Using these estimates we find that
the right-hand side of \eqref{I-1-1} is convergent absolutely if
\eqref{region3} and $\sigma_2>(\kappa+1)/2$ holds.

Therefore now we verify the decomposition \eqref{decompose} under the
condition \eqref{region4}.
In this region, combining with \eqref{f-i-e-sp} and \eqref{I-2-3}, we obtain
\begin{align}\label{basic}
L_2(s_1,s_2;{\frak A})=J_2(s_1,s_2;{\frak A})+L_1(s_1,s_2;{\frak A})
\end{align}
where
\begin{align}\label{J-2-def}
J_2(s_1,s_2;{\frak A})=\frac{(2\pi)^{s_1+s_2}}{\Gamma(s_1)\Gamma(s_2)}
\int_0^{\infty}f(iy)\int_0^{\infty}h(x+y)x^{s_1-1}y^{s_2-1}dxdy.
\end{align}

\section{\bf{Contour integration}}\label{sec-3}

In this section we show an infinite series expression of $J_2(s_1,s_2;{\frak A})$,
whose terms can be written in terms of confluent hypergeometric functions.

Let $\mathcal{C}$ be the contour which starts at $+\infty$, goes along the real
axis to a small positive number, rounds the origin counterclockwise, and then
goes back to $+\infty$ again along the real axis.
At first we assume \eqref{region4}.
Then, since $\sigma_1>0$, we can replace the inner integral
of \eqref{J-2-def} by the integral along $\mathcal{C}$ to obtain
\begin{align}\label{contour}
J_2(s_1,s_2;{\frak A})=
\frac{(2\pi)^{s_1+s_2}}{\Gamma(s_1)\Gamma(s_2)(e^{2\pi is_1}-1)}I_3,
\end{align}
where
\begin{align}\label{I-3-def}
I_3=\int_0^{\infty}f(iy)y^{s_2-1}\int_{\mathcal{C}}h(x+y)x^{s_1-1}dxdy.
\end{align}
The inner integral of \eqref{I-3-def} is absolutely convergent for any
$s_1$ with $\sigma_1<1$, and is $O(1)$ uniformly in $y$.   Therefore, using
\eqref{I-1-2}, we find that the double integral on the right-hand side of
\eqref{I-3-def} is absolutely convergent when
\begin{align}\label{region5}
\sigma_1<1,\quad \sigma_2>\frac{\kappa+1}{2}.
\end{align}
Our assumption (ii) implies that $L(s,{\frak A})$ is meromorphic in $\mathbb{C}$, and
$L_2(s_1,s_2;{\frak A})$ is meromorphic in the whole space $\mathbb{C}^2$,
as was mentioned in the introduction.
Therefore we can now
conclude that formula \eqref{basic} is valid in the region \eqref{region5}.

Next we replace the contour $\mathcal{C}$ by
$$
\mathcal{C}_R=\{x=-y+2\pi(R+1/2)e^{i\phi}\;|\;0\leq\phi<2\pi\}
$$
($R\in\mathbb{N}$), and let $R\to\infty$.
Since $h(x+y)=O(1)$ on $\mathcal{C}_R$ (\cite[formula (5.2)]{Mat98}),
we see that
$$
\int_{\mathcal{C}_R}h(x+y)x^{s_1-1}dx \to 0
$$
(as $R\to\infty$) if $\sigma_1<0$, which we now assume.   That is, we are now in
the subregion
\begin{align}\label{region6}
\sigma_1<0,\quad \sigma_2>\frac{\kappa+1}{2}
\end{align}
of \eqref{region5}.
Then by the residue calculus we have
\begin{align}\label{inner}
\int_{\mathcal{C}}h(x+y)x^{s_1-1}dx=-2\pi i\sum_{m\in\mathbb{Z},m\neq 0}{\rm Res}
     _{x=-y+im}\left(h(x+y)x^{s_1-1}\right),
\end{align}
and the value of the residue at $x=-y+im$ ($m\neq 0$) is given by
\begin{align}\label{residue}
&\lim_{\delta\to 0}\delta\left(\frac{1}{e^{2\pi(im+\delta)}-1}-\frac{1}
   {2\pi(im+\delta)}\right)(-y+im+\delta)^{s_1-1}\\
&=\lim_{\delta\to 0}\frac{1}{2\pi}\frac{2\pi\delta}{e^{2\pi\delta}-1}
    (-y+im+\delta)^{s_1-1}=\frac{1}{2\pi}(-y+im)^{s_1-1}.\notag
\end{align}
Substituting these results into \eqref{I-3-def} we obtain
\begin{align}\label{I-3-1}
I_3=-i\int_0^{\infty}f(iy)\sum_{m\in\mathbb{Z},m\neq 0}(-y+im)^{s_1-1}
       y^{s_2-1}dy.
\end{align}
When $m>0$, we see that
\begin{align}\label{m-posi}
(-y+im)^{s_1-1}=(e^{\pi i}y+e^{\pi i/2}m)^{s_1-1}
              =(e^{\pi i/2}m(z+1))^{s_1-1},
\end{align}
where $z$ is defined by $y=me^{-\pi i/2}z$.    Similarly, when $m<0$,
\begin{align}\label{m-nega}
(-y+im)^{s_1-1}=(e^{\pi i}y+e^{3\pi i/2}|m|)^{s_1-1}
              =(e^{3\pi i/2}|m|(z+1))^{s_1-1},
\end{align}
where $z$ is defined by $y=|m|e^{\pi i/2}z$.   Therefore, if the change of
integration and summation is possible, from \eqref{I-3-1} we obtain
\begin{align}\label{I-3-2}
I_3=-i(I_{31}+I_{32}),
\end{align}
where
\begin{align}\label{I-31-def}
I_{31}&=\sum_{m\geq 1}(e^{\pi i/2}m)^{s_1-1}\int_0^{i\infty}f(mz)
    (me^{-\pi i/2}z)^{s_2-1}(z+1)^{s_1-1}me^{-\pi i/2}dz\\
   &=\sum_{m\geq 1}e^{\pi i(s_1-s_2-1)/2}m^{s_1+s_2-1}\int_0^{i\infty}f(mz)
      z^{s_2-1}(z+1)^{s_1-1}dz\notag
\end{align}
and (rewriting $|m|$ as $m$)
\begin{align}\label{I-32-def}
I_{32}&=\sum_{m\geq 1}(e^{3\pi i/2}m)^{s_1-1}\int_0^{-i\infty}f(-mz)
    (me^{\pi i/2}z)^{s_2-1}(z+1)^{s_1-1}me^{\pi i/2}dz\\
   &=\sum_{m\geq 1}e^{\pi i(3s_1+s_2-3)/2}m^{s_1+s_2-1}\int_0^{-i\infty}f(-mz)
      z^{s_2-1}(z+1)^{s_1-1}dz.\notag
\end{align}
Substitute the definition of $f(mz)$ into \eqref{I-31-def} and change the
integration and summation again to obtain
$$
I_{31}=\sum_{m,n\geq 1}e^{\pi i(s_1-s_2-1)/2}m^{s_1+s_2-1}a(n)
    \int_0^{i\infty}e^{2\pi imnz}z^{s_2-1}(z+1)^{s_1-1}dz.
$$
Putting $mn=l$, this is equal to
\begin{align}\label{I-31-1}
e^{\pi i(s_1-s_2-1)/2}\sum_{l\geq 1}A_{s_1+s_2-1}^0(l)
     \int_0^{i\infty}e^{2\pi ilz}z^{s_2-1}(z+1)^{s_1-1}dz,
\end{align}
where
\begin{align}\label{A-def}
A_c^0(l)=\sum_{mn=l}m^c a(n).
\end{align}
Then we see that
the integral on the right-hand side of \eqref{I-31-1} is
$\Gamma(s_2)\Psi(s_2,s_1+s_2;-2\pi il)$
(because $\sigma_2>0$ is satisfied by
\eqref{region6}, and $\phi=\pi/2$ so $\phi+\arg(-2\pi il)=0$), hence
\begin{align}\label{I-31-2}
I_{31}=e^{\pi i(s_1-s_2-1)/2}\Gamma(s_2)\sum_{l\geq 1}A_{s_1+s_2-1}^0(l)
          \Psi(s_2,s_1+s_2;-2\pi il).
\end{align}
Similarly we obtain
\begin{align}\label{I-32-2}
I_{32}=e^{\pi i(3s_1+s_2-3)/2}\Gamma(s_2)\sum_{l\geq 1}A_{s_1+s_2-1}^0(l)
          \Psi(s_2,s_1+s_2;2\pi il).
\end{align}

To verify the above changing process (twice) of integration and summation,
we check the absolute convergence of the resulting expression.
Putting $lz=i\xi$, we see that the integral on the right-hand side of
\eqref{I-31-1} is
\begin{align}\label{I-31-3}
&=\int_0^{\infty}e^{-2\pi\xi}\left(\frac{i\xi}{l}\right)^{s_2-1}
         \left(1+\frac{i\xi}{l}\right)^{s_1-1}\frac{i}{l}d\xi\\
&\ll \int_0^{\infty}e^{-2\pi\xi}\left(\frac{\xi}{l}\right)^{\sigma_2-1}
       \left|\left(1+\frac{i\xi}{l}\right)^{s_1-1}\right|\frac{d\xi}{l}\notag,
\end{align}
where the implied constant depends on $s_2$.   Further,
\begin{align}\label{I-31-3bis}
\left|\left(1+\frac{i\xi}{l}\right)^{s_1-1}\right|
  &=\left|1+\frac{i\xi}{l}\right|^{\sigma_1-1}e^{-t_1\arg(1+i\xi/l)}.
\end{align}
The first factor on the right-hand side is $\leq 1$, because $|1+i\xi/l|\geq 1$
and $\sigma_1<0$ by
\eqref{region6}, while the second factor is $O_{t_1}(1)$ because
$|\arg(1+i\xi/l)|\leq\pi/2$.   Hence the right-hand side of \eqref{I-31-3} is
$$
\ll \int_0^{\infty}e^{-2\pi\xi}\left(\frac{\xi}{l}\right)^{\sigma_2-1}
     \frac{d\xi}{l}
\ll l^{-\sigma_2}.
$$
Therefore \eqref{I-31-1} is
\begin{align*}
&\ll\sum_{l\geq 1}|A_{s_1+s_2-1}^0(l)|l^{-\sigma_2}
\ll \sum_{m,n\geq 1}m^{\sigma_1+\sigma_2-1}|a(n)|(mn)^{-\sigma_2}\\
&\ll\sum_{m\geq 1}m^{\sigma_1-1}\sum_{n\geq 1}n^{(\kappa-1)/2+\varepsilon
   -\sigma_2},
\end{align*}
which is convergent in the region \eqref{region6}.   Therefore the whole step
of the above procedure is verified.

Define
\begin{align}\label{F-def}
F_{\pm}^0(s_1,s_2;{\frak A})=\sum_{l\geq 1}A_{s_1+s_2-1}^0(l)\Psi(s_2,s_1+s_2;
    \pm 2\pi il).
\end{align}
Using notation \eqref{F-def}, from \eqref{I-3-2}, \eqref{I-31-2} and
\eqref{I-32-2} we obtain
\begin{align}\label{I-3-3}
I_3&=-i\Gamma(s_2)\left\{e^{\pi i(s_1-s_2-1)/2}F_-^0(s_1,s_2;{\frak A})+
      e^{\pi i(3s_1+s_2-3)/2}F_+^0(s_1,s_2;{\frak A})\right\}\\
  &=i\Gamma(s_2)\left\{e^{\pi i(s_1-s_2+1)/2}F_-^0(s_1,s_2;{\frak A})+
      e^{\pi i(3s_1+s_2-1)/2}F_+^0(s_1,s_2;{\frak A})\right\}\notag
\end{align}
in the region \eqref{region6}.

Using the identity
\begin{align}\label{gamma}
\frac{1}{\Gamma(s_1)(e^{2\pi is_1}-1)}=\frac{\Gamma(1-s_1)}
   {2\pi ie^{\pi i s_1}}
\end{align}
we find that \eqref{contour} is rewritten as
\begin{align}\label{contour2}
J_2(s_1,s_2;{\frak A})=
\frac{(2\pi)^{s_1+s_2-1}\Gamma(1-s_1)}{ie^{\pi i s_1}\Gamma(s_2)}I_3.
\end{align}
Substituting \eqref{I-3-3} into the above, we obtain
\begin{align}\label{J-1}
J_2(s_1,s_2;{\frak A})
=&\frac{(2\pi)^{s_1+s_2-1}\Gamma(1-s_1)}{e^{\pi i s_1}}\\
  &\qquad  \times \left\{e^{\pi i(s_1-s_2+1)/2}F_-^0(s_1,s_2;{\frak A})+
      e^{\pi i(3s_1+s_2-1)/2}F_+^0(s_1,s_2;{\frak A})\right\}\notag\\
 =&(2\pi)^{s_1+s_2-1}\Gamma(1-s_1)\notag\\
  &\qquad  \times \left\{e^{\pi i(1-s_1-s_2)/2}F_-^0(s_1,s_2;{\frak A})+
      e^{\pi i(s_1+s_2-1)/2}F_+^0(s_1,s_2;{\frak A})\right\}\notag
\end{align}
in the region \eqref{region6}.
\section{\bf{Completion of the proof of Theorem \ref{thm-1}}}\label{sec-4}

The transformation formula
\begin{align}\label{transf}
\Psi(a,c;x)=x^{1-c}\Psi(a-c+1,2-c;x)
\end{align}
of the confluent hypergeometric function is well-known
(\cite[formula 6.5(6)]{Er}).   Using \eqref{transf}, we see that
\begin{align}\label{transf2}
&F_{\pm}^0(s_1,s_2;{\frak A})\\
&\quad=(2\pi e^{\pm \pi i/2})^{1-s_1-s_2}\sum_{l\geq 1}A_{s_1+s_2-1}^0(l)
    l^{1-s_1-s_2}\Psi(1-s_1,2-s_1-s_2;\pm 2\pi il).\notag
\end{align}
Since
$$
A_c^0(l)l^{-c}=\sum_{mn=l}m^c a(n)l^{-c}=\sum_{mn=l}\left(\frac{l}{m}\right)^{-c}
a(n)=\sum_{0<n|l}n^{-c}a(n)
$$
which is equal to $A_{-c}(l)$ (recall \eqref{A-def2}), we obtain
\begin{align}\label{F-0-F}
F_{\pm}^0(s_1,s_2;{\frak A})=(2\pi e^{\pm \pi i/2})^{1-s_1-s_2}
     F_{\pm}(1-s_2,1-s_1;{\frak A}).
\end{align}
Substituting this into \eqref{J-1}, we obtain
\begin{align}\label{J-2}
J_2(s_1,s_2;{\frak A})=\Gamma(1-s_1)\left\{F_-(1-s_2,1-s_1;{\frak A})
+F_+(1-s_2,1-s_1;{\frak A})\right\}
\end{align}
in the region \eqref{region6}.   Combining with \eqref{basic}, we now obtain
\eqref{thm-formula} in the region \eqref{region6}.

To complete the proof of the above theorem, it is sufficient to show that
$F_{\pm}(s_1,s_2;{\frak A})$ can be continued to the whole space $\mathbb{C}^2$.
For this purpose we use \eqref{asymp}.   This implies
\begin{align}\label{asymp2}
&\Psi(s_2,s_1+s_2;\pm 2\pi il)\\
&\quad=\sum_{j=0}^{M-1}\frac{(-1)^j (1-s_1)_j(s_2)_j}{j!}
    (\pm 2\pi il)^{-s_2-j}+\rho_M(s_2,s_1+s_2;\pm 2\pi il),\notag
\end{align}
so
\begin{align}\label{asymp3}
F_{\pm}(s_1,s_2;{\frak A})&=\sum_{j=0}^{M-1}\frac{(-1)^j (1-s_1)_j(s_2)_j}{j!}
    \sum_{l\geq 1}A_{s_1+s_2-1}(l)(\pm 2\pi il)^{-s_2-j}\\
   &+\sum_{l\geq 1}A_{s_1+s_2-1}(l)\rho_M(s_2,s_1+s_2;\pm 2\pi il).\notag
\end{align}
We see that
\begin{align}\label{asymp4}
&\sum_{l\geq 1}A_{s_1+s_2-1}(l)(\pm 2\pi il)^{-s_2-j}\\
&\quad=(\pm 2\pi i)^{-s_2-j}\sum_{l\geq 1}\sum_{n|l}n^{s_1+s_2-1}a(n)l^{-s_2-j}
          \notag\\
&\quad=(\pm 2\pi i)^{-s_2-j}\sum_{m,n\geq 1}n^{s_1+s_2-1}a(n)(mn)^{-s_2-j}
          \notag\\
&\quad=(\pm 2\pi i)^{-s_2-j}\sum_{m\geq 1}m^{-s_2-j}\sum_{n\geq 1}a(n)
     n^{s_1-1-j},\notag
\end{align}
whose last two sums are convergent when $\sigma_2>1-j$ and
$\sigma_1<j-(\kappa-1)/2$, and so the above
is equal to $(\pm 2\pi i)^{-s_2-j}\zeta(s_2+j)L(1-s_1+j,{\frak A})$.   Therefore
\begin{align}\label{asymp5}
&F_{\pm}(s_1,s_2;{\frak A})\\
&\;=\sum_{j=0}^{M-1}\frac{(-1)^j (1-s_1)_j(s_2)_j}{j!}(\pm 2\pi i)^{-s_2-j}
    \zeta(s_2+j)L(1-s_1+j,{\frak A})\notag\\
   &\quad+\sum_{l\geq 1}A_{s_1+s_2-1}(l)\rho_M(s_2,s_1+s_2;\pm 2\pi il).\notag
\end{align}
The explicit form of $\rho_M$ is
\begin{align}\label{rho}
\rho_M(a,c;x)&=\frac{(-1)^M(a-c+1)_M}{\Gamma(a)}\int_0^{e^{i\phi}\infty}
    e^{-xy}y^{a+M-1}\\
  &\quad\times\int_0^1 \frac{(1-\tau)^{M-1}}{(M-1)!}(1+\tau y)^{c-a-M-1}d\tau dy
   \notag
\end{align}
(see \cite[(3.3)]{Mat04}).   In the present situation, $\phi=\mp\pi/2$.
Therefore, if the change of integration and summation is possible, we have
\begin{align}\label{rho-2}
&\sum_{l\geq 1}A_{s_1+s_2-1}(l)\rho_M(s_2,s_1+s_2;\pm 2\pi il)\\
&\;=\frac{(-1)^M(1-s_1)_M}{\Gamma(s_2)}\int_0^{\mp i\infty}
   \sum_{l\geq 1}A_{s_1+s_2-1}(l)e^{\mp 2\pi ily}y^{s_2+M-1}\notag\\
&\quad\times \int_0^1 \frac{(1-\tau)^{M-1}}{(M-1)!}(1+\tau y)^{s_1-M-1}d\tau dy
   \notag\\
&\;=\frac{(-1)^M(1-s_1)_M}{\Gamma(s_2)}\int_0^{\infty}
   \sum_{l\geq 1}A_{s_1+s_2-1}(l)e^{-\eta}\left(\frac{\mp i\eta}{2\pi l}\right)
   ^{s_2+M-1}\notag\\
&\quad\times \int_0^1 \frac{(1-\tau)^{M-1}}{(M-1)!}\left(1\mp \frac{i\tau\eta}
    {2\pi l}\right)^{s_1-M-1}d\tau\frac{\mp i}{2\pi l}d\eta\notag\\
&\;=\frac{(-1)^M(1-s_1)_M}{\Gamma(s_2)}\int_0^{\infty}
     \sum_{l\geq 1}A_{s_1+s_2-1}(l)l^{-s_2-M}e^{-\eta}
     \left(\frac{\mp i\eta}{2\pi}\right)^{s_2+M-1}\notag\\
&\quad\times\int_0^1 \frac{(1-\tau)^{M-1}}{(M-1)!}\left(1\mp \frac{i\tau\eta}
    {2\pi l}\right)^{s_1-M-1}d\tau\frac{\mp i}{2\pi}d\eta.\notag
\end{align}
Similarly to \eqref{asymp4}, the sum with respect to $l$ is
$\zeta(s_2+M)L(1-s_1+M,{\frak A})$, if
\begin{align}\label{region7}
\sigma_1<M+1-\frac{\kappa+1}{2},\quad \sigma_2>1-M.
\end{align}
Also, similarly to the argument around \eqref{I-31-3bis},
we see that
$$
\left(1\mp \frac{i\tau\eta}{2\pi l}\right)^{s_1-M-1}=O(1)
$$
if $\sigma_1<M+1$.   Therefore the integral on the
right-hand side of \eqref{rho-2} is
$$
\ll_{s_1,s_2,M}\int_0^{\infty}e^{-\eta}
   \left|\left(\frac{\mp i\eta}{2\pi}\right)^{s_2+M-1}\right|d\eta
$$
which is convergent if $\sigma_2>-M$.
Hence the above change is verified, and
we can now conclude that the second sum
on the right-hand side of \eqref{asymp5} is convergent in the region
\eqref{region7}.   This implies that \eqref{asymp5} gives the meromorphic
continuation of $F_{\pm}(s_1,s_2;{\frak A})$ to the region \eqref{region7}.
Since $M$ is arbitrary, $F_{\pm}(s_1,s_2;{\frak A})$ can be continued
meromorphically to the whole space $\mathbb{C}^2$.
This completes the proof of Theorem \ref{thm-1}.

\begin{rmk}\label{rmk-symm}
In \cite{KMT-Debrecen} it was pointed out that Theorem \ref{thm-1} in the special
case ${\frak A}={\frak A}_1$ gives a symmetric form of the functional equation
on some hyperplanes
(\cite[Theorem 2.2]{KMT-Debrecen}).    It is desirable to deduce such a symmetric
form of the functional equation for general ${\frak A}$.
However there is a difficulty,
caused by the fact that $F_{\pm}^0$ on the left-hand side of \eqref{F-0-F} is
different from $F_{\pm}$ on the right-hand side.
\end{rmk}

\section{\bf{Modularity comes into play}}\label{sec-5}

Now we proceed to the proof of Theorem \ref{thm-2}.   Therefore, hereafter, we
assume that $\kappa$ is an even positive integer and $f(\tau)$ is a cusp form of
weight $\kappa$ with respect to $\Gamma_0(N)$.

In Section \ref{sec-3}, we only used the estimates \eqref{I-1-2} to show that
the double integral \eqref{I-3-def} is convergent in the region \eqref{region5}.
However, since now we assume that $f(\tau)$ is a cusp form, we see that $f(iy)$
is also of exponential decay when $y\to 0 $.   Therefore $I_3$ is convergent for
any $s_2\in\mathbb{C}$, and so we can remove the second condition
$\sigma_2>(\kappa+1)/2$ from \eqref{region5}, and also from \eqref{region6}.
This implies that, if $f$ is a cusp form, the whole argument in
Section \ref{sec-3} is valid in the region
\begin{equation}\label{region6bis}
\sigma_1<0
\end{equation}
instead of \eqref{region6}.   (This is important, because the former region
\eqref{region6} has no intersection with \eqref{region10} below.)

We begin with \eqref{I-31-def}, in the region \eqref{region6bis}.
Using \eqref{f-tilde}, we can rewrite \eqref{I-31-def} as
\begin{align}\label{I-31-mod}
I_{31}&=\sum_{m\geq 1}e^{\pi i(s_1-s_2-1)/2}m^{s_1+s_2-1}\int_0^{i\infty}
   (-\sqrt{N}mz)^{-\kappa}\widetilde{f}\left(-\frac{1}{Nmz}\right)z^{s_2-1}
(z+1)^{s_1-1}dz\\
&=N^{-\kappa/2}\sum_{m\geq 1}e^{\pi i(s_1-s_2-1)/2}m^{s_1+s_2-1-\kappa}
   \int_0^{i\infty}\widetilde{f}\left(-\frac{1}{Nmz}\right)z^{-\kappa+s_2-1}
   (z+1)^{s_1-1}dz,\notag
\end{align}
because $(-mz)^{-\kappa}=(mz)^{-\kappa}$ since $\kappa$ is even.
Putting $z=1/w$, we see that the above integral is
\begin{align*}
&=\int_{-i\infty}^0 \widetilde{f}\left(-\frac{w}{Nm}\right)w^{\kappa-s_2+1}
    \left(\frac{1}{w}+1\right)^{s_1-1}\left(-\frac{dw}{w^2}\right)\\
&=\int_0^{-i\infty}\sum_{n\geq 1}\widetilde{a}(n)e^{2\pi in(-w/Nm)}w^{\kappa-s_1-s_2}
     (w+1)^{s_1-1}dw,
\end{align*}
so, if we can change the integration and summation, we have
\begin{align}\label{I-31-mod2}
I_{31}&=N^{-\kappa/2}e^{\pi i(s_1-s_2-1)/2}\sum_{m,n\geq 1}m^{s_1+s_2-1-\kappa}
      \widetilde{a}(n)\\
  &\quad\times\int_0^{-i\infty}e^{-2\pi i(n/Nm)w}w^{\kappa-s_1-s_2}
  (w+1)^{s_1-1}dw.\notag
\end{align}
If
\begin{align}\label{region8}
\sigma_1+\sigma_2<\kappa+1
\end{align}
holds, then the last integral is expressed by the confluent hypergeometric
function (with $\phi=-\pi/2$ here); that is,
\begin{align}\label{I-31-mod3}
I_{31}&=N^{-\kappa/2}e^{\pi i(s_1-s_2-1)/2}\Gamma(\kappa-s_1-s_2+1)\sum_{m,n\geq 1}
      m^{s_1+s_2-1-\kappa}\widetilde{a}(n)\\
     &\quad\times
     \Psi(\kappa-s_1-s_2+1,\kappa-s_2+1;2\pi in/Nm)\notag\\
   &=N^{-\kappa/2}e^{\pi i(s_1-s_2-1)/2}\Gamma(\kappa-s_1-s_2+1)
      H_{2,N}^+(-s_1,\kappa-s_2+1;\widetilde{f}).\notag
\end{align}

We have to check the convergence of \eqref{I-31-mod2} in order to verify the
above interchange of integration and summation.   Put $2\pi i(n/Nm)w=y$.
Then the integral on the right-hand side of \eqref{I-31-mod2} is
\begin{align*}
&=\int_0^{\infty}e^{-y}\left(\frac{Nmy}{2\pi in}\right)^{\kappa-s_1-s_2}
   \left(\frac{Nmy}{2\pi in}+1\right)^{s_1-1}\frac{Nm}{2\pi in}dy\\
&=\left(\frac{Nm}{2\pi in}\right)^{\kappa-s_1-s_2+1}\int_0^{\infty}e^{-y}
   y^{\kappa-s_1-s_2}\left(\frac{Nmy}{2\pi in}+1\right)^{s_1-1}dy.
\end{align*}
Under \eqref{region6bis} we can show
$$
   \left(\frac{Nmy}{2\pi in}+1\right)^{s_1-1}=O(1)
$$
similarly to the argument around \eqref{I-31-3bis}.
But this is not sufficient to prove the convergence of the series
with respect to $m$.   We should be more careful here:
Using
$$
\left(\frac{Nmy}{2\pi in}+1\right)^{s_1-1}=\left(\frac{Nmy+2\pi in}{2\pi in}\right)
   ^{s_1-1},
$$
we obtain
\begin{align}\label{I-31-mod4}
I_{31}&=N^{-\kappa/2}e^{\pi i(s_1-s_2-1)/2}\sum_{m,n\geq 1}
      m^{s_1+s_2-1-\kappa}\widetilde{a}(n)\\
&\;\times (Nm)^{\kappa-s_1-s_2+1}(2\pi in)^{-\kappa+s_2}\notag\\
&\;\times\int_0^{\infty}e^{-y}y^{\kappa-s_1-s_2}(Nmy+2\pi in)^{s_1-1}dy.\notag
\end{align}
Denote the last integral by $J_{31}$.   Then
\begin{align}\label{I-31-mod5}
I_{31}\ll_N \sum_{m,n\geq 1}|\widetilde{a}(n)|
       n^{-\kappa+\sigma_2}|J_{31}|
\end{align}
(where $\ll_N$ means that the implied constant depends on $N$), and
\begin{align*}
J_{31}&=\int_0^{n/Nm}+\int_{n/Nm}^{\infty}\\
   &\ll\int_0^{n/Nm}e^{-y}y^{\kappa-\sigma_1-\sigma_2}n^{\sigma_1-1}dy
    +\int_{n/Nm}^{\infty}e^{-y}y^{\kappa-\sigma_1-\sigma_2}(Nmy)^{\sigma_1-1}dy\\
   &\ll n^{\sigma_1-1}\int_0^{n/Nm}e^{-y}y^{\kappa-\sigma_1-\sigma_2}dy
    +(Nm)^{\sigma_1-1}\int_{n/Nm}^{\infty}e^{-y}y^{\kappa-\sigma_2-1}dy\\
   &=n^{\sigma_1-1}J_{311}+(Nm)^{\sigma_1-1}J_{312},
\end{align*}
say.   Note that $J_{311}$ is convergent in the region \eqref{region8}, and
$J_{312}$ is always convergent.
As for $J_{312}$, we just use the following simple estimate:
\begin{align}\label{J-312}
J_{312}\leq \int_0^{\infty}e^{-y}y^{\kappa-\sigma_2-1}dy,
\end{align}
where the integral on the right-hand side is convergent (hence $O(1)$) if
\begin{align}\label{region9}
\sigma_2<\kappa
\end{align}
holds.   Consider $J_{311}$.   When $Nm\geq n$, we see that
\begin{align}\label{J-311-1}
J_{311}\leq \int_0^{n/Nm}y^{\kappa-\sigma_1-\sigma_2}dy
     \ll \left(\frac{n}{Nm}\right)^{\kappa-\sigma_1-\sigma_2+1}
\end{align}
(under \eqref{region8}), while when $Nm<n$ we have
\begin{align}\label{J-311-2}
J_{311}\leq \int_0^{\infty}e^{-y}y^{\kappa-\sigma_1-\sigma_2}dy\ll 1
\end{align}
(under \eqref{region8}).   Therefore, under the conditions \eqref{region8} and
\eqref{region9}, we have
\begin{align}\label{J-31}
J_{31}\ll\left\{
  \begin{array}{lll}
   n^{\sigma_1-1}\left(\frac{n}{Nm}\right)^{\kappa-\sigma_1-\sigma_2+1}
   +(Nm)^{\sigma_1-1}  &  {\rm if}  &  Nm\geq n,\\
   n^{\sigma_1-1}+(Nm)^{\sigma_1-1} &  {\rm if}  &  Nm<n.
  \end{array}\right.
\end{align}
Substituting this into \eqref{I-31-mod5}, we obtain
\begin{align}\label{I-31-mod6}
I_{31}&\ll_N
\sum_{Nm\geq n}|\widetilde{a}(n)|n^{-\kappa+\sigma_2}
          \left((Nm)^{\sigma_1+\sigma_2-\kappa-1}n^{\kappa-\sigma_2}
         +(Nm)^{\sigma_1-1}\right)\\
&\;+\sum_{Nm<n}|\widetilde{a}(n)|n^{-\kappa+\sigma_2}
       \left(n^{\sigma_1-1}+(Nm)^{\sigma_1-1}\right)\notag\\
&\;=I_{31}^*+I_{31}^{**},\notag
\end{align}
say.
Using Deligne's estimate, we find that
\begin{align}\label{mmmmm}
I_{31}^*\ll \sum_{m\geq 1}(Nm)^{\sigma_1+\sigma_2-\kappa-1}\sum_{n\leq Nm}
      n^{(\kappa-1)/2+\varepsilon}
    +\sum_{m\geq 1}(Nm)^{\sigma_1-1}\sum_{n\leq Nm}
      n^{(\kappa-1)/2+\varepsilon-\kappa+\sigma_2}.
\end{align}
The first double sum is
\begin{align*}
\ll \sum_{m\geq 1}(Nm)^{\sigma_1+\sigma_2-\kappa-1+(\kappa+1)/2
      +\varepsilon}
\ll_N \sum_{m\geq 1}m^{\sigma_1+\sigma_2-(\kappa+1)/2+\varepsilon},
\end{align*}
which is convergent if
\begin{align}\label{region9bis}
\sigma_1+\sigma_2<\frac{\kappa-1}{2}.
\end{align}
The inner sum of the second double sum of \eqref{mmmmm} is
$O((Nm)^{\sigma_2-(\kappa-1)/2+\varepsilon})$ if $\sigma_2\geq (\kappa-1)/2$,
and $O(1)$ otherwise.    Therefore the second double sum is convergent if
$\sigma_2\geq (\kappa-1)/2$ and \eqref{region9bis} holds, or if
$\sigma_2< (\kappa-1)/2$ and $\sigma_1<0$.
Consequently we find that the right-hand side of \eqref{mmmmm} is convergent
in the region
\begin{align}\label{region9bb}
\sigma_1<0,\quad \sigma_1+\sigma_2<\frac{\kappa-1}{2}.
\end{align}

Similarly we find that
\begin{align}\label{nnnnn}
I_{31}^{**}&\ll \sum_{n\geq 1}n^{(\kappa-1)/2+\varepsilon-\kappa+\sigma_2+\sigma_1-1}
      \sum_{Nm< n}1+\sum_{n\geq 1}n^{(\kappa-1)/2+\varepsilon-\kappa+\sigma_2}
      \sum_{Nm<n}(Nm)^{\sigma_1-1}
\end{align}
whose first double sum is
$$
\ll \sum_{n\geq 1}n^{\sigma_1+\sigma_2-(\kappa+1)/2+\varepsilon}
$$
which is convergent in the region \eqref{region9bis}.
The second double sum converges in the same region if $\sigma_1\geq 0$, while if
$\sigma_1<0$ it is convergent when $\sigma_2<(\kappa-1)/2$.
Therefore the right-hand side of \eqref{nnnnn} is convergent in the region
\begin{align}\label{region9bbb}
\sigma_2<\frac{\kappa-1}{2},\quad \sigma_1+\sigma_2<\frac{\kappa-1}{2}.
\end{align}

Therefore by \eqref{region8}, \eqref{region9}, \eqref{region9bb} and
\eqref{region9bbb}, we now
arrive at the conclusion that the right-hand side of \eqref{I-31-mod2} is
absolutely convergent in the region
\begin{align}\label{region10}
\sigma_1<0,\quad \sigma_2<\frac{\kappa-1}{2}.
\end{align}

\bigskip

Next we consider $I_{32}$.   Using the modularity again, similarly to
\eqref{I-31-mod}, we have
\begin{align}\label{I-32-mod}
I_{32}&=\sum_{m\geq 1}e^{\pi i(3s_1+s_2-3)/2}m^{s_1+s_2-1}\int_0^{-i\infty}
   (\sqrt{N}mz)^{-\kappa}\widetilde{f}\left(\frac{1}{Nmz}\right)z^{s_2-1}
   (z+1)^{s_1-1}dz\\
&=N^{-\kappa/2}\sum_{m\geq 1}e^{\pi i(3s_1+s_2-3)/2}m^{s_1+s_2-1-\kappa}\int_0^{-i\infty}
   \widetilde{f}\left(\frac{1}{Nmz}\right)z^{-\kappa+s_2-1}(z+1)^{s_1-1}dz.\notag
\end{align}
Putting $z=1/w$, similarly to \eqref{I-31-mod2}, we obtain
\begin{align}\label{I-32-mod2}
I_{32}&=N^{-\kappa/2}e^{\pi i(3s_1+s_2-3)/2}\sum_{m,n\geq 1}m^{s_1+s_2-1-\kappa}
   \widetilde{a}(n)\\
  &\;\times\int_0^{i\infty}e^{2\pi i(n/Nm)w}w^{\kappa-s_1-s_2}
  (w+1)^{s_1-1}dw\notag\\
&=N^{-\kappa/2}e^{\pi i(3s_1+s_2-3)/2}\Gamma(\kappa-s_1-s_2+1)\sum_{m,n\geq 1}
     m^{s_1+s_2-1-\kappa}\widetilde{a}(n)\notag\\
  &\;\times\Psi(\kappa-s_1-s_2+1,\kappa-s_2+1;-2\pi in/Nm)\notag\\
&=N^{-\kappa/2}e^{\pi i(3s_1+s_2-3)/2}\Gamma(\kappa-s_1-s_2+1)
    H_{2,N}^-(-s_1,\kappa-s_2+1;\widetilde{f}).\notag
\end{align}
The convergence can be discussed exactly the same way as in the case of $I_{31}$.
(This time we start with $-2\pi i(n/m)w=y$.)   Hence \eqref{I-32-mod2} is also
valid in the region \eqref{region10}.

Substituting \eqref{I-31-mod3} and \eqref{I-32-mod2} into \eqref{I-3-2},
and combining with \eqref{contour2}, we have
\begin{align}\label{J-H}
J_2(s_1,s_2;f)&=-\frac{(2\pi)^{s_1+s_2-1}\Gamma(1-s_1)}{e^{\pi is_1}\Gamma(s_2)}
    \Gamma(\kappa-s_1-s_2+1)\\
&\qquad\times N^{-\kappa/2}
 \left\{e^{\pi i(s_1-s_2-1)/2}H_{2,N}^+(-s_1,\kappa-s_2+1;\widetilde{f})\right.\notag\\
&\;\qquad \left.+e^{\pi i(3s_1+s_2-3)/2}H_{2,N}^-(-s_1,\kappa-s_2+1;\widetilde{f})
       \right\}\notag\\
&=\frac{(2\pi)^{s_1+s_2-1}\Gamma(1-s_1)}{\Gamma(s_2)}
    \Gamma(\kappa-s_1-s_2+1)\notag\\
&\qquad\times N^{-\kappa/2}\left\{e^{\pi i(1-s_1-s_2)/2}
  H_{2,N}^+(-s_1,\kappa-s_2+1;\widetilde{f})\right.\notag\\
&\;\qquad \left.+e^{\pi i(s_1+s_2-1)/2}
  H_{2,N}^-(-s_1,\kappa-s_2+1;\widetilde{f})\right\}.\notag
\end{align}
This with \eqref{basic} gives \eqref{thm2-formula} in the region \eqref{region10}.
Therefore, to complete the proof of Theorem \ref{thm-2}, the remaining task is
to show the meromorphic continuation of $H_{2,N}^{\pm}$.


\section{\bf{The meromorphic continuation of $H_{2,N}^{\pm}(s_1,s_2;\widetilde{f})$}}
\label{sec-6}

In this section we prove that $H_{2,N}^{\pm}(s_1,s_2;\widetilde{f})$ can
be continued meromorphically to the whole space $\mathbb{C}^2$.   We first consider
the case $H_{2,N}^+$.    Applying \eqref{hyp-def} (with $\phi=-\pi/2$) to the
right-hand side of \eqref{H-2-def} and putting $y=-i\eta$, we obtain
\begin{align}\label{H-7-1}
H_{2,N}^+(s_1,s_2;\widetilde{f})&=\frac{-i}{\Gamma(s_1+s_2)}\sum_{m,n\geq 1}
   m^{-s_1-s_2}\widetilde{a}(n)\int_0^{\infty}e^{-2\pi(n/Nm)\eta}\\
&\times(-i\eta)^{s_1+s_2-1}(-i\eta+1)^{-s_1-1}d\eta.\notag
\end{align}

The argument in the preceding section shows that the double series form of
$H_{2,N}^+(-s_1,\kappa-s_2+1;\widetilde{f})$
(that is, \eqref{I-31-mod3})
is absolutely convergent in the region
\eqref{region10}.   This implies that the right-hand side of \eqref{H-7-1} is
absolutely convergent in the region
\begin{align}\label{region11}
\sigma_1>0,\quad \sigma_2>\frac{\kappa+3}{2}.
\end{align}
Therefore, if we assume \eqref{region11}, we may change the order of summation and
integration on the right-hand side of \eqref{H-7-1} to obtain
\begin{align}\label{H-7-2}
&H_{2,N}^+(s_1,s_2;\widetilde{f})\\
&=\frac{-i}{\Gamma(s_1+s_2)}\int_0^{\infty}
   \sum_{m\geq 1}m^{-s_1-s_2}\widetilde{f}\left(\frac{i\eta}{Nm}\right)
   (-i\eta)^{s_1+s_2-1}(-i\eta+1)^{-s_1-1}d\eta\notag\\
&=\frac{-i}{\Gamma(s_1+s_2)}\int_0^{\infty}\widetilde{\mathcal{F}}(i\eta,s_1+s_2)
   (-i\eta)^{s_1+s_2-1}(-i\eta+1)^{-s_1-1}d\eta,\notag
\end{align}
where
\begin{align}\label{cal-F-def}
\widetilde{\mathcal{F}}(\tau,s)=\sum_{m\geq 1}\widetilde{f}
   \left(\frac{\tau}{Nm}\right)m^{-s}.
\end{align}

\begin{lem}\label{mellin}
Let $u$ be a complex variable.   We have
\begin{align}\label{mellin-formula}
\int_0^{\infty}\widetilde{\mathcal{F}}(i\eta,s)\eta^{u-1}d\eta=\Gamma(u)
\left(\frac{N}{2\pi}\right)^u\zeta(s-u)L(u,\widetilde{f})
\end{align}
in the region
\begin{align}\label{region12}
\Re(s)-1>\Re(u)>\frac{\kappa+1}{2}.
\end{align}\end{lem}

\begin{proof}
We have
\begin{align*}
&\int_0^{\infty}\widetilde{\mathcal{F}}(i\eta,s)\eta^{u-1}d\eta\\
&=\sum_{m,n\geq 1}m^{-s}\widetilde{a}(n)\int_0^{\infty}e^{-2\pi(n/Nm)\eta}
    \eta^{u-1}d\eta\\
&=\sum_{m,n\geq 1}m^{-s}\widetilde{a}(n)\Gamma(u)\left(\frac{Nm}{2\pi n}\right)^u\\
&=\Gamma(u)\left(\frac{N}{2\pi}\right)^u\sum_{m\geq 1}m^{-s+u}\sum_{n\geq 1}
    \widetilde{a}(n)n^{-u}\\
&=\Gamma(u)\left(\frac{N}{2\pi}\right)^u\zeta(s-u)L(u,\widetilde{f}),
\end{align*}
where in the above calculations, changes of summation and integration can
be verified because of absolute convergence under condition \eqref{region12}.
\end{proof}

Therefore, $\widetilde{\mathcal{F}}(i\eta,s)$ is the inverse Mellin transform
of the right-hand side of \eqref{mellin-formula}, and hence
\begin{align}\label{F-mod1}
\widetilde{\mathcal{F}}(i\eta,s_1+s_2)=\frac{1}{2\pi i}\int_{(c)}
   \eta^{-u}\Gamma(u)\left(\frac{N}{2\pi}\right)^u\zeta(s_1+s_2-u)L(u,\widetilde{f})du,
\end{align}
where $c=\Re(u)$ satisfies
\begin{align}\label{condition-c}
\sigma_1+\sigma_2-1>c>\frac{\kappa+1}{2}
\end{align}
and the path of integration is the vertical line from $c-i\infty$ to $c+i\infty$.

From \eqref{H-7-2} and \eqref{F-mod1} we obtain
\begin{align}\label{H-7-3}
&H_{2,N}^+(s_1,s_2;\widetilde{f})
=\frac{-1}{2\pi\Gamma(s_1+s_2)}\int_0^{\infty}\int_{(c)}
   \eta^{-u}\Gamma(u)\left(\frac{N}{2\pi}\right)^u\zeta(s_1+s_2-u)L(u,\widetilde{f})du\\
&\qquad\qquad\times (-i\eta)^{s_1+s_2-1}(-i\eta+1)^{-s_1-1}d\eta\notag\\
&\qquad=\frac{-1}{2\pi\Gamma(s_1+s_2)}\int_{(c)}
   \Gamma(u)\left(\frac{N}{2\pi}\right)^u\zeta(s_1+s_2-u)L(u,\widetilde{f})J(u)du,\notag
\end{align}
where
\begin{align}\label{J-def}
J(u)=\int_0^{\infty}\eta^{-u}(-i\eta)^{s_1+s_2-1}(-i\eta+1)^{-s_1-1}d\eta,
\end{align}
if the change of the order of integration is possible.
The integral \eqref{J-def} is absolutely convergent in the region
\begin{align}\label{region13}
\sigma_2<c+1,\quad \sigma_1+\sigma_2>c,
\end{align}
and hence the above change of integration is valid by Fubini's theorem in this
region.
From \eqref{condition-c} and \eqref{region13} we see that \eqref{H-7-3} is valid
in the region
\begin{align}\label{region13b}
\sigma_2<c+1,\quad \sigma_1+\sigma_2-1>c>\frac{\kappa+1}{2}.
\end{align}
Since the intersection of \eqref{region11} and \eqref{region13b} is non-empty,
now we find that $H_{2,N}^+(s_1,s_2;\widetilde{f})$ is continued to the region
\eqref{region13b} by the expression \eqref{H-7-3}.

Putting $y=-i\eta$ on the right-hand side of \eqref{J-def}, and rotating the path of
integration to the positive real axis (this is possible under condition
\eqref{region13}), we obtain
$$
J(u)=e^{\pi i(1-u)/2}\int_0^{\infty}y^{s_1+s_2-1-u}(1+y)^{-s_1-1}dy.
$$
Therefore, applying the beta integral formula we obtain
\begin{align}\label{J-mod1}
J(u)=e^{\pi i(1-u)/2}\frac{\Gamma(u-s_2+1)\Gamma(s_1+s_2-u)}{\Gamma(s_1+1)}.
\end{align}
Substituting this into \eqref{H-7-3}, we now arrive at the expression
\begin{align}\label{H-7-4}
H_{2,N}^+(s_1,s_2;\widetilde{f})=&\frac{-1}
{2\pi \Gamma(s_1+s_2)\Gamma(s_1+1)}
\int_{(c)}\Gamma(u)\Gamma(u-s_2+1)\Gamma(s_1+s_2-u)\\
&\times e^{\pi i(1-u)/2}\left(\frac{N}{2\pi}\right)^u
\zeta(s_1+s_2-u)L(u,\widetilde{f})du\notag
\end{align}
in the region \eqref{region13}.

We prove that the right-hand side of \eqref{H-7-4} can be continued meromorphically
to the whole space $\mathbb{C}^2$ by suitable modifications of the path of
integration.   Let $(s_1^0,s_2^0)$ be any point in the space $\mathbb{C}^2$.
We choose a point $(s_1^*,s_2^*)$ in the region \eqref{region13}, which satisfies
$\Im s_1^*=\Im s_1^0$, $\Im s_2^*=\Im s_2^0$.
Then \eqref{H-7-4} holds for $(s_1,s_2)=(s_1^*,s_2^*)$.   The poles of the integrand
on the right-hand side of \eqref{H-7-4} are

(A) $u=0,-1,-2,-3,\ldots$,

(B) $u=s_2^*-1,s_2^*-2,s_2^*-3,\ldots$,

(C) $u=s_1^*+s_2^*,s_1^*+s_2^*+1,s_1^*+s_2^*+2,\ldots$,

(D) $u=s_1^*+s_2^*-1$.

\noindent
The poles (A) and (B) are on the left of the vertical line $\Re u=c$,
while the poles (C) and (D) are on the right of $\Re u=c$.

First consider the case when $\Im(s_1^*+s_2^*)\neq\Im s_2^*$ and
$\Im(s_1^*+s_2^*)\neq 0$.
Let $L_1$ be the line segment joining $s_2^*-1$ and $s_2^0-1$, and $L_2$ the line
segment joining $s_1^*+s_2^*-1$ and $s_1^0+s_2^0-1$.   We deform the original
path $\Re u=c$ to make a new path $\mathcal{D}$, such that $L_1$ is on the left of
$\mathcal{D}$ while $L_2$ is on the right of $\mathcal{D}$ (see Fig.1).

\begin{center}
\begin{psfrags}
\psfrag{a}{$s_2^*-1$}
\psfrag{b}{$s_2^0-1$}
\psfrag{c}{$s_1^0+s_2^0-1$}
\psfrag{d}{$s_1^*+s_2^*-1$}
\psfrag{0}{$0$}
\psfrag{-1}{$-1$}
\psfrag{-2}{$-2$}
\psfrag{z}{$c$}
\includegraphics[scale=1]{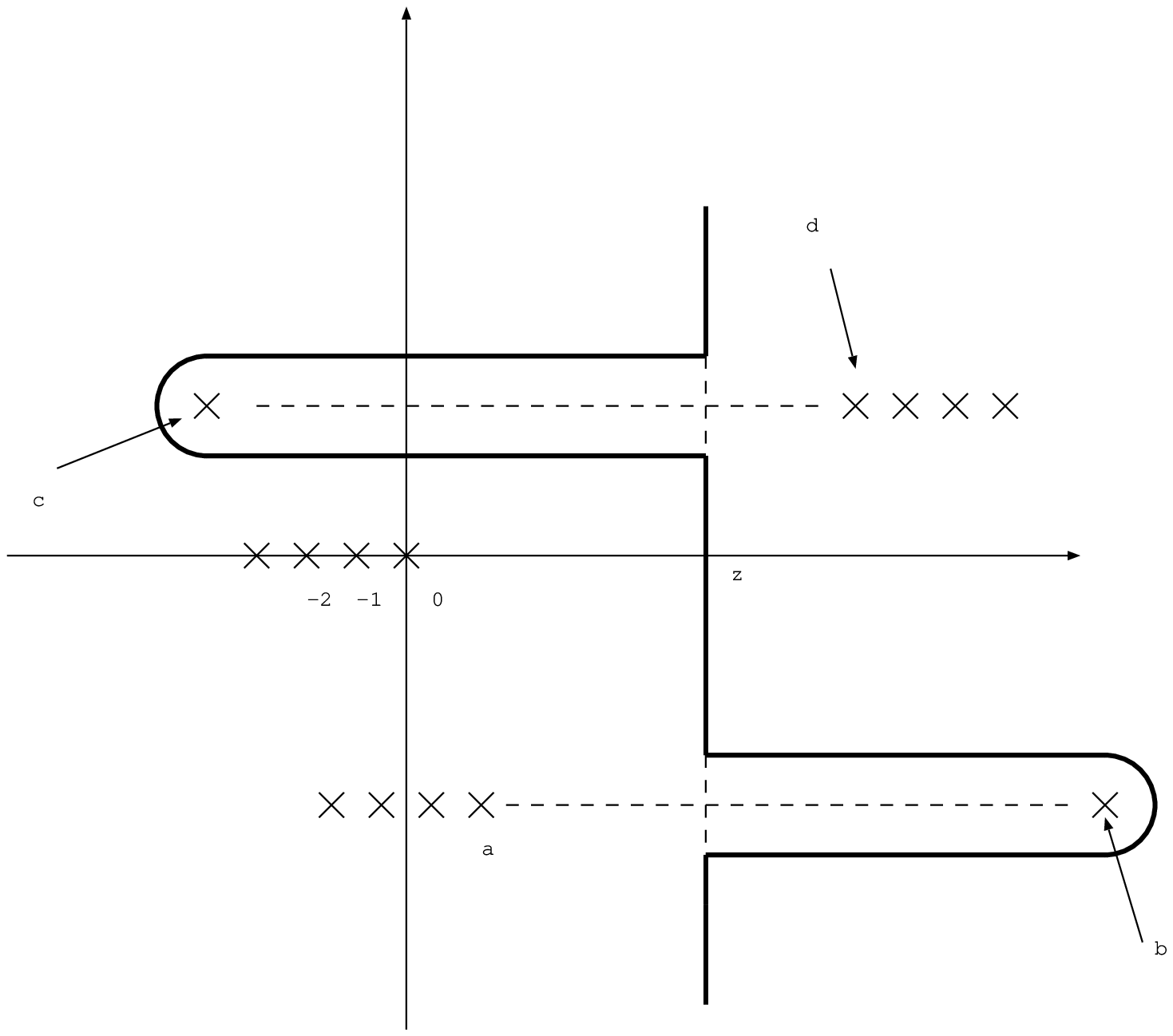}
\end{psfrags}

Fig.1
\end{center}

Then we move the variables $(s_1,s_2)$ from $(s_1^*,s_2^*)$ to $(s_1^0,s_2^0)$,
keeping the values of their imaginary parts.   Then the location of poles moves,
but during this procedure they do not encounter the new path $\mathcal{D}$.
Therefore in this case we can continue $H_{2,N}^+(s_1,s_2;\widetilde{f})$ to
the point $(s_1^0,s_2^0)$ holomorphically.

Next consider the situation when $\Im(s_1^*+s_2^*)=\Im s_2^*$ or
$\Im(s_1^*+s_2^*)= 0$.   We discuss the former case, because the latter case
can be treated similarly.   When $\Im(s_1^*+s_2^*)=\Im s_2^*$, we deform
$\Re u=c$ to make $\mathcal{D}'$, which only requires that $L_1$ is on the left of
$\mathcal{D}'$.   Then, when $s_1^*+s_2^*-1$ is moved to $s_1^0+s_2^0-1$, several
poles encounter $\mathcal{D}'$.   Therefore at the point $(s_1^0,s_2^0)$, we have to
add the residue terms coming from the above poles to the right-hand side of
\eqref{H-7-4}.   This new expression of $H_{2,N}^+(s_1,s_2;\widetilde{f})$ gives
the meromorphic continuation to the point $(s_1^0,s_2^0)$.   We thereby obtain the
proof of meromorphic continuation of $H_{2,N}^+(s_1,s_2;\widetilde{f})$ to the
whole space $\mathbb{C}^2$.

The function $H_{2,N}^-(s_1,s_2;\widetilde{f})$ can be treated quite similarly.
We can show the integral expression of $H_{2,N}^-(s_1,s_2;\widetilde{f})$, almost
the same as \eqref{H-7-4}, only the factor $e^{\pi i(1-u)/2}$ is replaced by
$e^{\pi i(u-1)/2}$, which can be continued meromorphically as above.
The proof of Theorem \ref{thm-2} is now complete.

\begin{rmk}\label{rmk-sing}
We can discuss the location of singularities of
$H_{2,N}^{\pm}(s_1,s_2;\widetilde{f})$ more closely, by using the more
sophisticated path (like the path $\mathcal{C}'$ defined in \cite{Mat06},
described in Fig.2 of \cite{Mat06}) instead of $\mathcal{D}'$.
\end{rmk}

\section{\bf{Proof of Corollary \ref{cor-2}}}\label{sec-7}

We conclude this paper with the proof of Corollary \ref{cor-2}.

The right-hand side of the formula \eqref{thm2-formula}
given in Theorem \ref{thm-2} 
consists of two terms, the term $L_1(s_1,s_2;f)$ and the other.

First consider the term $L_1(s_1,s_2;f)$.
The denominator $\Gamma(s_2)$ on the right-hand side of the
definition \eqref{def-G} of $L_1(s_1,s_2;f)$ has poles at
$s_2=-l$, and the other factors do not cancel those poles.
Therefore $s_2=-l$ are zero-divisors of $L_1(s_1,s_2;f)$. 

Consider the other term on the right-hand side of 
\eqref{thm2-formula}.   Again the denominator is $\Gamma(s_2)$,
so what we have to show is that the other factors do not cancel
the poles $s_2=-l$ of $\Gamma(s_2)$.    This is obvious
except for the terms
$H_{2,N}^{\pm}(-s_1,\kappa-s_2+1;\widetilde{f})$.
As for $H_{2,N}^{\pm}$, in Section \ref{sec-6} we noticed that
these are absolutely convergent, hence especially finite, in the region
\eqref{region10}.    Since the region \eqref{region10} includes
$s_2=-l$ when $\sigma_1<0$, we now arrive at the conclusion that
$s_2=-l$ are zero-divisors of the second term on the right-hand side
of \eqref{thm2-formula}.



\end{document}